\documentclass[12pt]{amsart}
\usepackage[T2A]{fontenc}
\usepackage[utf8]{inputenc}   
\usepackage[russian, english]{babel}
\usepackage{enumerate}

\usepackage{float}

\usepackage{mathrsfs}  

\usepackage{eucal}

\usepackage{hyperref}
\usepackage{amsmath}
\usepackage{amsthm}
\usepackage{amsfonts}
\usepackage{amssymb}
\usepackage{cite}

\usepackage{mathptmx} 
\usepackage[scaled=0.90]{helvet} 
\usepackage{courier} 
\usepackage[T1]{fontenc}
\usepackage{color}
\usepackage{bbm}

\usepackage[dvips]{graphicx}
\DeclareGraphicsExtensions{.eps}

\newtheorem{theorem}{Theorem}[section]
\newtheorem{utv*}{Proposition}
\newtheorem{hyp*}{Conjecture}
\newtheorem*{example*}{Example}
\newtheorem{lemma}[theorem]{Lemma}
\newtheorem{corollary}[theorem]{Corollary}

\newtheorem*{th*}{Theorem}
\newtheorem{prop}[theorem]{Proposition}
\theoremstyle{definition}

\newtheorem{defin}[theorem]{Definition}
\newtheorem{example}[theorem]{Example}
\theoremstyle{remark}
\newtheorem*{rem*}{Remark}

\def\sli{\sum\limits}
\def\ili{\int\limits}

\def\R{\mathbb{R}}

\def\ep{\varepsilon}
\def\vf{\varphi}

\newcommand{\diam}{\operatorname{diam}}

\newcommand{\dist}{\operatorname{dist}}

\def\cyr{\fontencoding{OT2}\fontfamily{wncyr}\selectfont}
\DeclareTextFontCommand{\textcyr}{\cyr}

{\end{list}}

\newcounter{vremennyj}

\def\t{\tilde}
\numberwithin{equation}{section}
\begin{document}
\title[Local properties of minimal energy points]{Local properties of Riesz minimal energy configurations and equilibrium measures}
\author{D. P. Hardin}
\address{Center for Constructive Approximation, Department of Mathematics, Vanderbilt University}
\email{doug.hardin@vanderbilt.edu}
\author{A. Reznikov}
\email{aleksandr.b.reznikov@vanderbilt.edu}
\author{E. B. Saff}
\thanks{The research of D. P. Hardin, A. Reznikov and E. B. Saff was supported, in part, by National Science Foundation grant DMS-1516400}
\email{edward.b.saff@vanderbilt.edu}
\author{A. Volberg}
\address{Department of Mathematics, Michigan State University}
\email{volberg@math.msu.edu}
\thanks{The research of A. Volberg was supported, in part, by National Science Foundation grant DMS-1600065}
\date{\today}
\begin{abstract}
We investigate separation properties of $N$-point configurations that minimize discrete Riesz $s$-energy on a compact set $A\subset \R^p$. When $A$ is a smooth $(p-1)$-dimensional manifold without boundary and $s\in [p-2, p-1)$, we prove that the order of separation (as $N\to \infty$) is the best possible. The same conclusions hold for the points that are a fixed positive distance from the boundary of $A$ whenever $A$ is any $p$-dimensional set. These estimates extend a result of Dahlberg for certain smooth $(p-1)$-dimensional surfaces when $s=p-2$ (the harmonic case). Furthermore, we obtain the same separation results for `greedy' $s$-energy points. 
We deduce our results from an upper regularity property of the $s$-equilibrium measure (i.e., the measure that solves the continuous minimal Riesz $s$-energy problem), and we show that this property holds under a local smoothness assumption on the set $A$.
\end{abstract}
\maketitle
\section{Introduction}

In this paper we study, respectively, the properties of separation and regularity for minimal discrete and for continuous Riesz energy. For a measure $\mu$ supported on a compact set $A$ in Euclidean space and $s>0$, its {\it Riesz $s$-potential} and {\it Riesz $s$-energy} are defined by
\begin{equation}\label{potendefin}
U_s^\mu(x):=\int_A \frac{\textup{d}\mu(y)}{|x-y|^s}, \;\;\;\; I_s[\mu]:=\int_A U_s^\mu(x)\textup{d}\mu(x), 
\end{equation}
and its {\it Riesz $\log$-potential} and {\it Riesz $\log$-energy} by
$$
U_{\log}^\mu(x):=\int_A \log\frac{1}{|x-y|}\textup{d}\mu(y), \;\;\;\; I_{\log}[\mu]:=\int_A U_{\log}^\mu(x)\textup{d}\mu(x).
$$
The constant $W_s(A):=\inf I_s[\mu]$, where the infimum is taken over all probability measures $\mu$ supported on $A$, is called the {\it $s$-Wiener constant} of the set $A$, and the {\it $s$-capacity} of $A$ is given by
$$
\text{cap}_s(A):=\frac{1}{W_s(A)}, \;\; s>0, \qquad \text{cap}_{\log}(A):=\exp(-W_{\log}(A)).
$$
If $W_s(A)<\infty$, it is known that there exists a unique probability measure $\mu_s$ that attains $W_s(A)$ and we call $\mu_s$ the {\it $s$-equilibrium measure for $A$} (see \cite{Landkof1972}). 

The problem of minimizing $I_s[\mu]$ has a discrete analog. Namely, for an integer $N\geqslant 2$ we set
$$
\mathcal{E}_s(A, N):=\min_{\omega_N\subset A} E_s(\omega_N),
$$
where the infimum is taken over all $N$-point configurations $\omega_N=\{x_1, \ldots, x_N\}\subset A$ and
$$
E_s(\omega_N):=\sli_{i\not = j} \frac{1}{|x_i - x_j|^s}.
$$
By $\omega_N^*=\omega_{N, s}^*=\{x_1^*, \ldots, x_N^*\}$ we denote any {\it optimal $N$-point $s$-energy configuration}; i.e., a configuration that attains $\mathcal{E}_s(A, N)$. 
It is known that if $W_s(A)<\infty$, then 
$$
\frac{1}{N}\sli_{j=1}^N \delta_{x^*_j}\stackrel{*}{\to} \mu_s,
$$
where $\delta_x$ denotes the unit point mass at $x$, and the convergence is in the weak$^*$ topology. Thus, for sets of positive $s$-capacity, by solving the discrete minimization problem, we ``discretize'' the measure $\mu_s$ that solves the continuous problem. 

We shall study properties of $\omega^*_N$, especially its {\it separation distance} given by 
\begin{equation}
\delta(\omega^*_N):=\min_{i\not = j} |x^*_i-x^*_j|.
\end{equation}
In the theory of approximation and interpolation, the separation distance is often associated with some measure of stability of the approximation. In \cite{Dahlberg1978} Dahlberg proved that for a $C^{1+\epsilon}$-smooth $d$-dimensional manifold $A\subset \R^{d+1}$ without boundary and $s=d-1$ (the harmonic case), there exists a constant $c>0$ such that
\begin{equation}\label{eqintro1}
\delta(\omega^*_N)\geqslant cN^{-1/d}, \;\;\; \forall \; N\geqslant 2.
\end{equation}
For such a set $A$, the order $N^{-1/d}$ for separation of $N$-point configurations is best possible\footnote[3]{More generally, this is true for any set $A$ that is lower $d$-regular with respect to some finite measure $\mu$ (see Definition \ref{defregular}).}. For the special case $A=\mathbb{S}^d:=\{x\in \R^{d+1}\colon |x|=1\}$, Kuijlaars, Saff and Sun \cite{Kuijlaars2007} extended Dahlberg's result by proving \eqref{eqintro1} for $s\in [d-1, d)$ and in \cite{Brauchart2014}, Brauchart, Dragnev and Saff extended the range of $s$ to $s\in (d-2, d)$ with explicit values for the constant $c$. Our first goal is to extend the results from \cite{Dahlberg1978} and \cite{Kuijlaars2007} to all $C^\infty$-smooth $d$-dimensional manifolds for $s\in [d-1, d)$ and to interior points of $d$-dimensional bodies for $s\in (d-2, d)$. More generally, we show that \eqref{eqintro1} holds whenever the $s$-equilibrium measure of the manifold is upper regular (see Theorem \ref{corenergysepar}).

Since the problem of determining the minimum $\mathcal{E}_s(A, N)$ requires solving an extremal problem in $N$ variables, it is natural to consider a somewhat simpler discretization method, namely, the computation of {\it greedy $s$-energy points} defined below which involves minimization in only a single variable. For the logarithmic kernel on $A\times A$ where $A\subset\mathbb{C}$, such points were introduced by Edrei \cite{edrei1939} and extensively explored by Leja \cite{leja1957} and his students. For general kernels they were investigated by L\'opez and Saff \cite{Lopezgreedy}.
\begin{defin}\label{defgreedy}
A sequence $\omega^*_\infty=\{a^*_j\}_{j=1}^\infty\subset A$ is called a {\it sequence of greedy $s$-energy points} if $a^*_1\in A$ and for every $N>1$ we have
$$
\sli_{j=1}^{N-1}\frac{1}{|a^*_N-a^*_j|^s} = \inf_{y\in A}\sli_{j=1}^{N-1}\frac{1}{|y-a^*_j|^s}.
$$
\end{defin}
Notice that if $\omega_{N-1}:=\{a^*_1, \ldots, a^*_{N-1}\}$ is already determined, then $a^*_N$ is chosen to minimize $E_s(\omega_{N-1}\cup \{y\})$ over all $y\in A$. It is known \cite{Lopezgreedy} that if $W_s(A)<\infty$ and $\omega_\infty^*=\{a_j^*\}_{j=1}^\infty$ is a sequence of greedy $s$-energy points, then
$$
\frac{1}{N}\sli_{j=1}^N \delta_{a^*_j}\stackrel{*}{\to} \mu_s.
$$
Some computational aspects of using the greedy $s$-energy points for numerical integration can be found in \cite{elefante}. Our second goal, which is achieved in Theorem \ref{corgreedy111} and Corollary \ref{cornonint2}, is to prove that for a smooth $d$-dimensional manifold $A$ and $s\in (d-1, d)$ or $s>d$, there exists a constant $c>0$ such that, for every $i<j$, we have
$$
|a^*_i-a^*_j|\geqslant cj^{-1/d}.
$$
In particular, this implies that $\delta(\{a_1^*, \ldots, a_N^*\})\geqslant cN^{-1/d}$. Moreover, when $s>d$ we also prove that for some constant $C>0$ the {\it covering radius} $\eta$ for such point satisfies
$$
\eta(\{a_1^*, \ldots, a_N^*\}, A):=\max_{y\in A} \min_{j=1,\ldots, N} |y-a^*_j|\leqslant CN^{-1/d}.
$$
For configurations that attain the minimal discrete energy $\mathcal{E}_s(A, N)$, this was done in \cite{Hardin2012} for $s>d$ and in \cite{Dahlberg1978} for $s=d-1$.

Since the method of proof for the above results utilizes the regularity properties of the measure $\mu_s$ (see Definition \ref{defregular}), our third goal is to obtain sufficient conditions for this regularity. As we show in Theorem \ref{thequil}, compact $C^\infty$-smooth $d$-dimensional manifolds $A\subset \R^{d+1}$ without boundary satisfy our conditions (we anticipate, however, that the same result holds for $C^2$-smooth manifolds). In the case $s=d-1$, such a result is proved in \cite{sjogren1972}. Another result of this type was proved in \cite{Wallin1966} under an assumption that the potential $U_s^\mu$ of the measure $\mu$ satisfies an appropriate H\"older condition in the whole space $\R^{d+1}$.
We derive our result, Theorem \ref{thequil}, using only smoothness of the manifold $A$ by applying the theory of pseudo-differential operators. 

The paper is organized as follows. The main results in the integrable case, which include separation properties of minimal energy and greedy energy points, are stated in Section \ref{secmainintegr} and proved in Sections \ref{sectionproofs1} and \ref{sectionproofs2}. In Section \ref{secmainnonint} we state the separation and covering properties of greedy energy points in the non-integrable case, which are proved in Subsections \ref{proofnonint1} and \ref{proofnonint2}. In Section \ref{factspotential} we cite some known results from potential theory that we need to prove our main results, and in Section \ref{factspseudo} we give a short introduction to the theory of pseudo-differential operators, which we need for the proof of Theorem \ref{thequil} in Section \ref{sectionproofs2}.

\section{Main results in the integrable case}\label{secmainintegr}
In this section we state and discuss our main results for integrable Riesz kernels. Their proofs are given in Sections \ref{sectionproofs1} and \ref{sectionproofs2}. We shall work primarily with a class of $\ell$-regular sets, which are defined as follows.
\begin{defin}\label{defregular}
A compact set $A$ is called {\it $\ell$-regular}, $\ell>0$, if for some measure $\lambda$ supported on $A$ there exists a positive constant $C$ such that for any $x\in A$ and $r<\diam(A)$ we have
$$
C^{-1}r^\ell \leqslant \lambda(B(x,r)) \leqslant Cr^\ell,
$$
where $B(x,r)$ denotes the open ball $B(x,r):=\{y\in \R^p\colon |y-x|<r\}$. The set $A$ is called {\it $\ell$-regular at $x\in A$} if for some positive number $r_1$, the set $A\cap B(x, r_1)$ is $\ell$-regular.

Further, we call a measure $\mu$ {\it upper $d$-regular at $x$} if for some constant $c(x)$ and any $r>0$ we have
\begin{equation}\label{defupperregular}
\mu(B(x,r))\leqslant c(x)r^d.
\end{equation}
\end{defin}
As the next example shows, a set $A$ can be $\ell$-regular with $\ell\in \mathbb{N}$, but its $s$-equilibrium measure $\mu_s$ can be $d$-regular with $d<\ell$. 
\begin{example}\label{exampleee}
For the closed unit ball $\mathbb{B}^\ell:=\{x\in \R^\ell\colon |x|\leqslant 1\}$, which is $\ell$-regular, and $s\in (\ell-2, \ell)$ the $s$-equilibrium measure is given by (see, e.g., \cite{Landkof1972} or \cite{Borodachov2016})
$$
\textup{d}\mu_s=M(1-|x|^2)^{(s-\ell)/2}\textup{d}x, \;\;\;\;\; M=\frac{\Gamma(1+s/2)}{\pi^{\ell/2}\Gamma(1+(s-\ell)/2)}.
$$
We notice that $\mu_s$ is $\ell$-regular at every interior point of $\mathbb{B}^\ell$.
However, for $x$ on the boundary $\partial \mathbb{B}^\ell=\mathbb{S}^{\ell-1}$, the measure $\mu_s$ satisfies 
$$
C^{-1}r^{(\ell+s)/2}\leqslant \mu_s(B(x, r))\leqslant Cr^{(\ell+s)/2},
$$
so that $\mu_s$ is not $\ell$-regular at $x\in \partial \mathbb{B}^\ell$. 
\end{example}

We now present our main results which include the possibility of different regularities for the set $A$ and the measure $\mu_s$. Although stated only for $s>0$, they remain valid for $\ell=1$ and $s=\log$.
\begin{theorem}\label{corenergysepar}
Let $A\subset \R^p$ be a compact set of positive $s$-capacity, $0\leqslant p-2<s<d\leqslant \ell\leqslant p$, and $\mu_s$ be the $s$-equilibrium measure on $A$. Assume $A$ is $\ell$-regular at every $x\in A'\subset A$ and $\mu_s$ is upper $d$-regular at every $x\in A'$ with $\sup_{x\in A'}c(x)\leqslant c$ for some $c>0$. Then there exists a positive constant $C$ such that for any optimal $N$-point $s$-energy configuration $\omega^*_N=\{x^*_1, \ldots, x^*_N\}$, any $x^*_j \in A'$ and any $x^*_k\in A$ with $k\not=j$ we have
\begin{equation}\label{zlatovlaska}
|x^*_j-x^*_k|>CN^{-1/d}.
\end{equation}
\end{theorem}

In particular, \eqref{zlatovlaska} holds in the following cases (see Corollaries \ref{thequilbodies} and \ref{corsupersmooth} and Example \ref{exampleee}):
\begin{itemize}
\item $A\subset \R^{\ell+1}$ is a compact $\ell$-regular $C^\infty$-smooth manifold without boundary, $s\in [\ell-1, \ell)$, and $A'=A$ with $d=\ell$;
\item $A\subset \R^\ell$ is compact, $s\in (\ell-2, \ell)$, and $A'=\{x\in A\colon \dist(x, \partial A)\geqslant \ep\}$ with $\ep>0$ and $d=\ell$;
\item $A=\mathbb{B}^{\ell}$, $s\in (\ell-2, \ell)$, and $A'=\{x\in \R^{\ell}\colon |x|\leqslant 1-\ep\}$ with $\ep\in(0,1)$ and $d=\ell$;
\item $A=\mathbb{B}^\ell$, $s\in (\ell-2, \ell)$, and $A'=\partial \mathbb{B}^\ell$ with $d=(s+\ell)/2$.
\end{itemize}

\begin{rem*}
In the case $\ell=1$ and $s=\log$, our results imply the sharp estimate that when $x^*_j = \pm 1$ and $x^*_k\not=x_j^*$,
\begin{equation}\label{sharp12345}
|x^*_k - x^*_j|\geqslant cN^{-2}.
\end{equation}
Indeed, in this case the optimal $\log$-energy configurations $\omega_N^*$ consist of {\it Fekete points}; i.e., the roots of $(1-x^2)P'_{N-1}(x)$, where $P_{N}$ is the $N$th degree Legendre polynomial (see, e.g., \cite{szego}), for which it is known that \eqref{sharp12345} cannot be improved for $x_k^*$ near $\pm 1$.

\end{rem*}
The next theorem concerns greedy energy points defined in Definition \ref{defgreedy}. 

\begin{theorem}\label{corgreedy111}
Let $A\subset \R^{\ell+1}$ be a compact $C^\infty$-smooth $\ell$-dimensional manifold without boundary, $\ell-1\leqslant s<\ell$. If $\omega^*_\infty = \{a^*_j\}_{j=1}^\infty$ is a sequence of greedy $s$-energy points on $A$, then there exists a positive constant $c(A, s)$ such that, for any $i<j$,
$$
|a^*_i - a^*_j|\geqslant c(A, s)j^{-1/\ell}.
$$
\end{theorem}
Theorems \ref{corenergysepar} and \ref{corgreedy111} are immediate consequences of Theorem \ref{dregularimpliessepar} stated below and the following trivial observation: if $\omega^*_N=\{x^*_1, \ldots, x^*_N\}$ is an optimal $N$-point $s$-energy configuration, then for any $k=1,\ldots, N$ we have
$$
\sum_{j\not = k} \frac{1}{|x^*_k-x^*_j|^s} = \inf_{y\in A}\sum_{j\not = k} \frac{1}{|y-x^*_j|^s}.
$$

\begin{theorem}\label{dregularimpliessepar}
Let $A\subset \R^p$ be a compact set of positive $s$-capacity and $\mu_s$ be the $s$-equilibrium measure on $A$. Let $\omega_N=\{x_1, \ldots, x_N\}$ be any $N$-point configuration in $A$, and $y^*\in A$ satisfy\footnote[2]{The right-hand side of \eqref{suddenlypolar} is called the {\it $s$-polarization} (see, e.g., \cite{borodachov2016optimall}) of $\omega_N$.}
\begin{equation}\label{suddenlypolar}
\sum_{j=1}^N \frac{1}{|y^*-x_j|^s} = \inf_{y\in A}\sum_{j=1}^N \frac{1}{|y-x_j|^s}.
\end{equation}
If \,$0\leqslant p-2<s<d\leqslant \ell\leqslant p$, $A$ is $\ell$-regular at $y^*$ and $\mu_s$ is upper $d$-regular at $y^*$, then for every $j=1, \ldots, N$
\begin{equation}\label{antoshka}
|y^*-x_j|\geqslant (c_1 c(y^*)+1)^{-1/s}\cdot N^{-1/d},
\end{equation}
where the constant $c(y^*)$ is from \eqref{defupperregular} and the positive constant $c_1$ depends only on $A$ and $s$.
\end{theorem}
Our next goal is to present a sufficient condition for Theorem \ref{dregularimpliessepar} to hold. We begin with the following definition.
\begin{defin}
Let $A\subset \R^p$ be a compact set $d$-regular at a point $x_0\in A$. We say that $A$ is {\it $(d, C^\infty)$-smooth at $x_0$} if there exists a positive number $r_0$ and a $C^\infty$-smooth invertible function $\varphi\colon B(x_0, r_0)\cap A \to \R^d$ such that $\vf(B(x_0, r_0)\cap A)$ is open in $\R^d$ and $\varphi^{-1}$ is also $C^\infty$-smooth.
\end{defin}
Our next theorem is a local result showing that if a manifold is $C^\infty$-smooth at a point, then the $s$-equilibrium measure is upper $d$-regular at this point.
\begin{theorem}\label{thequil}
Let $A\subset \R^{p}$ be a compact set of positive $s$-capacity, where  $p\in \{d, d+1\}$ and $s\in [p-2, d)$, and $\mu_s$ be the $s$-equilibrium measure on $A$. If $A$ is $(d, C^\infty)$-smooth at a point $x_0\in A$, then $\mu_s$ is upper $d$-regular at $x_0$; i.e., inequality \eqref{defupperregular} holds for any $r>0$.
\end{theorem}
Example \ref{exampleee} illustrates the sharpness of this theorem. We note that if $y^*$ is as in \eqref{suddenlypolar} and the assumptions of Theorem \ref{thequil} hold with $x_0$ replaced by $y^*$, then the conclusion of Theorem \ref{dregularimpliessepar} follows. 

The next corollary follows from Theorem \ref{thequil} and the fact that, if $p=d$, then $A$ is $(p, C^\infty)$-smooth at $x_0\in A$ if and only if $x_0$ is an interior point of $A$.
\begin{corollary}\label{thequilbodies}
Let $A\subset \R^d$ be compact, $s\in [d-2, d)$ and $x_0$ be an interior point of $A$. If $\mu_s$ is the $s$-equilibrium measure on $A$, then $\mu_s$ is upper $d$-regular at $x_0$.
\end{corollary}

Obviously, a $C^\infty$-smooth manifold without boundary satisfies the conditions of Theorem \ref{thequil}; therefore, we have the following consequence.
\begin{corollary}\label{corsupersmooth}
Let $A\subset \R^{d+1}$ be a compact $C^\infty$-smooth $d$-dimensional manifold without boundary, $d-1\leqslant s<d$ and $\mu_s$ be the $s$-equilibrium measure on $A$. Then $\mu_s$ is uniformly upper $d$-regular on $A$.
\end{corollary}

\section{Main results in the non-integrable case}\label{secmainnonint}
In this section we state an analog of Theorem \ref{dregularimpliessepar} for the case $s>d$ under very weak assumptions on the set $A$. As a consequence, we deduce separation and covering properties of greedy energy points in this case. These properties are proved in Section \ref{sectionproofs1}. Below $\mathcal{H}_d$ denotes the {\it $d$-dimensional Hausdorff measure} normalized by $\mathcal{H}_d([0,1]^d)=1$. By $\overline{\mathcal{M}}_d$ we denote the {\it upper $d$-dimensional Minkowskii content}; i.e., for a compact set $A\subset \R^p$, set
\begin{equation}\label{defincontent}
\overline{\mathcal{M}}_d(A):=\limsup_{\ep\to 0^+} \frac{\mathcal{L}_p\left( \{x\in \R^p\colon \dist(x, A)<\ep\}\right)}{\beta_{p-d}\ep^{p-d}},
\end{equation}
where $\mathcal{L}_p$ is the Lebesgue measure on $\R^p$ and $\beta_{p-d}$ is the volume of a $(p-d)$-dimensional unit ball (for $p=d$, we set $\beta_0$:=1).
\begin{prop}\label{thnonint1}
If $A\subset \R^p$ is a compact set with $\mathcal{H}_d(A)>0$ ($d\leqslant p$) and $s>d$, then there exists a constant $c>0$ such that for any $N$-point configuration $\omega_N=\{x_1, \ldots, x_N\}\subset A$ and $y^*\in A$ satisfying
$$
\sum_{j=1}^N \frac{1}{|y^*-x_j|^s} = \inf_{y\in A}\sum_{j=1}^N \frac{1}{|y-x_j|^s},
$$
we have, for every $j=1,\ldots, N$,
\begin{equation}\label{antoshka111}
|y^*-x_j|\geqslant c\cdot N^{-1/d}.
\end{equation}
\end{prop}
\begin{corollary}\label{cornonint2}
With the assumptions of Theorem \ref{thnonint1}, there exists a constant $c>0$ such that for any sequence $\omega^*_\infty=\{a^*_j\}_{j=1}^\infty$ of greedy energy points and any $i<j$, we have
\begin{equation}\label{eqnonintgreedysepar}
|a^*_i-a^*_j|\geqslant cj^{-1/d}.
\end{equation}
If, in addition, $A\subset \tilde{A}$ for a $d$-regular set $\tilde{A}$ and $\overline{\mathcal{M}}_d(A)<\infty$, then for some $c>0$ and every $N\geqslant 2$, the covering radius of $\omega^*_N:=\{a^*_1, \ldots, a^*_N\}\subset \omega^*_\infty$ satisfies
\begin{equation}\label{eqnonintgreedycover}
\eta(\omega^*_N, A)=\max_{y\in A}\min_{j=1,\ldots, N} |y-a_j^*|\leqslant cN^{-1/d}.
\end{equation}

\end{corollary}

\section{Some facts from Potential theory}\label{factspotential}
For the convenience of the reader we state several known results from potential theory that will be used in the proofs of the above formulated theorems. The following theorem can be found, for example, in \cite[p. 136]{Landkof1972} or \cite[Theorems 4.2.15 and 4.5.11]{Borodachov2016}.
\begin{theorem}\label{landkofblabla}
If $A\subset \R^p$ is a compact set of positive $s$-capacity, then the $s$-equilib\-ri\-um measure $\mu_s$ is unique. Moreover, the inequality $U_s^{\mu_s}(x)\leqslant W_s(A)$ holds $\mu_s$-a.e. and the inequality $U_s^{\mu_s}(x)\geqslant W_s(A)$ holds $s$-quasi-everywhere; i.e., if $F\subset\{x\in A \colon U_s^\mu(x) < W_s(A)\}$ is compact, then $W_s(F)=\infty$. Furthermore, if $s\in [p-2, p)$, then $U_s^{\mu_s}(x)\leqslant W_s(A)$ for every $x\in \R^p$.

\end{theorem}

The following theorem is a special case of \cite[Theorem 2.5]{reznikov2016minimum}.
\begin{theorem}\label{coralwaysequal}
Let $s<d$ and $\mu$ be a measure supported on $A\subset \R^p$, where $A$ is $d$-regular. If for some constant $M$ the inequality $U_s^\mu(x)\geqslant M$ holds $s$-quasi-everywhere on $A$, then it holds everywhere on $A$.  
\end{theorem}

We conclude this section with two results from the theory of non-integrable Riesz potentials. The first result can be found in \cite[Theorem 2.4]{Erdelyi2013} and \cite[Proposition 2.5]{Borodachov2007}, while the second is a consequence of the proof of \cite[Theorem 3]{Hardin2012}. 
\begin{theorem}\label{therdel}
Assume $A\subset \R^p$, $\mathcal{H}_d(A)>0$ and $s>d$. Then there exists two positive constants $c_1(s)$ and $c_2(s)$ such that for any $N$-point configuration $\omega_N=\{x_1, \ldots, x_N\} \subset A$ we have
$$
\inf_{y\in A}\sli_{j=1}^N \frac{1}{|y-x_j|^s} \leqslant c_1(s)N^{s/d}
$$
and
$$
E_s(\omega_N) = \sli_{i\not = j} \frac{1}{|x_i-x_j|^s} \geqslant c_2(s) \overline{\mathcal{M}}_d(A)^{-s/d} N^{1+s/d}.
$$
\end{theorem}
\begin{theorem}\label{thpotentiallast}
Suppose the compact set $A\subset \R^p$ with $\mathcal{H}_d(A)>0$ is contained in some $d$-regular compact set $\tilde{A}$ and $s>d$. If $\omega_N=\{x_1, \ldots, x_N\}\subset A$ is an $N$-point configuration with separation distance $\delta(\omega_N)\geqslant \tau N^{-1/d}$ for some $\tau>0$, then for some constant $R(s, \tau, p_s)$,
\begin{equation}
\eta(\omega_N, A):=\max_{y\in A}\min_{j=1,\ldots, N} |y-x_j|\leqslant R(s, \tau, p_s) N^{-1/d},
\end{equation}
where $p_s$ is any positive constant such that 
\begin{equation}\label{pspsps}
\inf_{y\in A} \sli_{j=1}^N \frac{1}{|y-x_j|^s} \geqslant p_s N^{s/d}.
\end{equation}
\end{theorem}

\section{Proofs of Theorem \ref{dregularimpliessepar} and Theorem \ref{thnonint1}}\label{sectionproofs1} 
For $x=(x(1), \ldots, x(p))\in A$, set $x_r:=(x(1), \ldots, x(p), r)\in \R^{p+1}$ and consider $A$ as a subset of $\R^{p+1}$ with $x=x_0$; i.e., $x(p+1)=0$.

The next lemma is related to results of Carleson \cite{Carleson1963} for $s\in [d-1, d)$ and Wallin \cite{Wallin1966}.
\begin{lemma}\label{potentialsecondterm}
Assume the measure $\mu$ on $A$ is upper $d$-regular at $x\in A$. If $d-2<s<d$, then there exists a constant $c_1$ that depends only on $s$ and $d$ such that 
$$
U_s^\mu(x_r)\geqslant U_s^\mu(x)-c_1 \cdot c(x)\cdot r^{d-s}.
$$
\end{lemma}
\begin{proof}
We first notice that for $x,y\in A$ we have $|y-x_r|^2 = |y-x|^2+r^2$. Therefore,
\begin{multline}\label{eqqq}
U_s^\mu(x)-U_s^\mu(x_r) = \int_A \frac{(|y-x|^2+r^2)^{s/2} - |y-x|^s}{(|y-x|^2+r^2)^{s/2} \cdot |y-x|^s} \text{d}\mu(y)  \\
=\ili_{|y-x|\leqslant 2r}\frac{(|y-x|^2+r^2)^{s/2} - |y-x|^s}{(|y-x|^2+r^2)^{s/2} \cdot |y-x|^s} \text{d}\mu(y) \\ + \ili_{|y-x|>2r}\frac{(|y-x|^2+r^2)^{s/2} - |y-x|^s}{(|y-x|^2+r^2)^{s/2} \cdot |y-x|^s} \text{d}\mu(y) =: I_1 + I_2.
\end{multline}

We have 
\begin{multline}\label{estint1}
I_1\leqslant \ili_{|y-x|\leqslant 2r} \frac{\textup{d}\mu(y)}{|y-x|^s} = \int_{0}^\infty \mu\{y\colon |y-x|\leqslant 2r, \; |y-x|^{-s}>t\}\text{d}t \\
= \int_{0}^{(2r)^{-s}} \mu\{y\colon |y-x|\leqslant 2r\} \text{d}t + \int_{(2r)^{-s}}^\infty \mu\{y\colon |y-x|<t^{-1/s}\}\text{d}t \\
\leqslant c(x) (2r)^{d-s} + c(x) \frac{s}{d-s} (2r)^{d-s} = 2^{d-s}\cdot \frac{d}{d-s} \cdot c(x) \cdot r^{d-s} = c_2 \cdot c(x) \cdot r^{d-s},
\end{multline}
where the constant $c_1$ depends only on $s$ and $d$.

To estimate $I_2$ we need the following inequality. For every positive $t$ there exists a constant $c$, such that for every $\ep<1/4$ we have
$$
(1+\ep)^t \leqslant 1+c\ep.
$$
This estimate is trivial since the function $\ep \mapsto ((1+\ep)^t -1)/\ep$ is continuous on the closed interval $[0, 1/4]$. Therefore,
\begin{multline}\label{estint2}
I_2 = \ili_{|y-x|>2r}\frac{(|y-x|^2+r^2)^{s/2} - |y-x|^s}{(|y-x|^2+r^2)^{s/2} \cdot |y-x|^s} \text{d}\mu(y) \\
\leqslant cr^2 \ili_{|y-x|>2r} \frac{\text{d}\mu(y)}{|y-x|^{s+2}} \leqslant cr^2 \ili_{0}^{(2r)^{-s-2}} \mu\{y\colon |y-x|<t^{-1/(s+2)}\}\text{d}t \\
\leqslant c_3\cdot c(x) \cdot r^2 \ili_{0}^{(2r)^{-s-2}} t^{-d/(s+2)}\text{d}t = c_4 \cdot c(x) \cdot r^{d-s}.
\end{multline}
Equality \eqref{eqqq} combined with estimates \eqref{estint1} and \eqref{estint2} imply the lemma.
\end{proof}
\subsection{Proof of Theorem \ref{dregularimpliessepar}}
Set 
$$
\gamma_N:=\sum_{j=1}^N \frac{1}{|y^*-x_j|^s} = \inf_{y\in A}\sum_{j=1}^N \frac{1}{|y-x_j|^s}.
$$
Since by Theorem \ref{landkofblabla} we have $U_s^{\mu_s}(x)\leqslant W_s(A)$ for every $x\in \R^p$, we deduce that
\begin{equation}\label{erdest}
\gamma_N \leqslant W_s(A) N .
\end{equation}
Setting $\nu(\omega_N):=\frac{1}{N}\sum_{j=1}^N \delta_{x_j}$, we obtain for $y\in A$ that
$$
U_s^{\nu(\omega_N)}(y)\geqslant \frac{1}{N}\frac{\gamma_N}{W_s(A)} W_s(A) \geqslant \frac{1}{N}\frac{\gamma_N}{W_s(A)} U_s^{\mu_s}(y),
$$
which by the domination principle for potentials (see \cite{ito64}) and Lemma \ref{potentialsecondterm} implies for $r:=N^{-1/d}$ that
\begin{equation}\label{etomine}
U_s^{\nu(\omega_N)}(y_r^*)\geqslant \frac{1}{N} \frac{\gamma_N}{W_s(A)} U_s^{\mu_s}(y^*_r) \geqslant  \frac{1}{N} \frac{\gamma_N}{W_s(A)} \left(U_s^{\mu_s}(y^*) - c_1\cdot c(y^*) N^{-1+s/d}\right).
\end{equation}
By Theorem \ref{coralwaysequal} and the $\ell$-regularity of $A$ at $y^*$, $U_s^{\mu_s}(y^*)\geqslant W_s(A)$; thus, it follows from \eqref{erdest} and \eqref{etomine} that
$$
U_s^{\nu(\omega_N)}(y_r^*) \geqslant \frac{\gamma_N}{N} - c_1 \cdot c(y^*) N^{-1+s/d},
$$
or
$$
\sli_{j=1}^N \frac{1}{|y_r^*-x_j|^s} \geqslant \gamma_N - c_1 \cdot c(y^*) N^{s/d}.
$$
Without loss of generality, we prove \eqref{antoshka} for $j=1$. Since $|y_r^*-x_1|\geqslant r=N^{-1/d}$ and $|y_r^*-x|\geqslant |y-x|$ for every $x\in A$, we have
\begin{multline}
\gamma_N - c_1 \cdot c(y^*) N^{s/d} \leqslant \sli_{j=1}^N \frac{1}{|y_r^*-x_j|^s} = \sli_{j=2}^N \frac{1}{|y_r^*-x_j|^s} + \frac{1}{|y_r^*-x_1|^s} \\
\leqslant \sli_{j=2}^N \frac1{|y^*-x_j|^s} + N^{s/d} = \sli_{j=1}^N \frac1{|y^*-x_j|^s} - \frac{1}{|y^*-x_1|^s} + N^{s/d} = \gamma_N -\frac{1}{|y^*-x_1|^s} + N^{s/d}. 
\end{multline}
Therefore,
$$
|y^*-x_1| \geqslant (c_1 c(y^*) + 1)^{-1/s} \cdot N^{-1/d}.
$$
\hfill \qed

\subsection{Proof of Proposition \ref{thnonint1}}\label{proofnonint1}
The proof is immediate. We merely observe that, by Theorem \ref{therdel} we have for every $j=1,\ldots,N$,
$$
c_1(s)N^{s/d}\geqslant \sli_{j=1}^N \frac{1}{|y^*-x_j|^s} \geqslant |y^*-x_j|^{-s};
$$
therefore,
$$
|y^*-x_j|\geqslant c_1(s)^{-1/s}N^{-1/d}.
$$
\hfill \qed
\subsection{Proof of Corollary \ref{cornonint2}}\label{proofnonint2}
We notice that the estimate \eqref{eqnonintgreedysepar} follows from Proposition \ref{thnonint1} and the fact that for every $j$ we have
$$
\sli_{i=1}^{j-1}\frac{1}{|a^*_j-a^*_i|^{s}} = \inf_{y\in A}\sli_{i=1}^{j-1}\frac{1}{|y-a^*_i|^{s}}.
$$
In view of inequality \eqref{pspsps} in Theorem \ref{thpotentiallast}, to deduce \eqref{eqnonintgreedycover} it is enough to show that the inequality
\begin{equation}\label{pspsps222}
\inf_{y\in A} \sli_{j=1}^N \frac{1}{|y-a^*_j|^s} \geqslant p_s N^{s/d}
\end{equation}
holds for some positive constant $p_s$ independent of $N$. For this purpose, observe that Theorem \ref{therdel} implies that for some positive $c$ that does not depend on $N$ we have, for $\omega_N=\{a_1^*, \ldots, a_N^*\}$,
\begin{equation}\label{nezadavali}
E_s(\omega_N)\geqslant c N^{1+s/d}.
\end{equation}
Hence, for every $j=1,\ldots, N$,
$$
\sli_{i=1}^{j-1}\frac{1}{|a^*_j-a^*_i|^s}=\inf_{y\in A}\sli_{i=1}^{j-1}\frac{1}{|y-a^*_i|^s} \leqslant \sli_{i=1}^{j-1}\frac{1}{|a^*_N-a^*_i|^s}\leqslant \sli_{i=1}^{N-1}\frac{1}{|a^*_N-a^*_i|^s},
$$
and so
$$
E_s(\omega_N) = 2\sli_{j=2}^N \sli_{i=1}^{j-1} \frac{1}{|a^*_j-a^*_i|^s} \leqslant 2N \sli_{i=1}^{N-1}\frac{1}{|a^*_N-a^*_i|^s} = 2N \inf_{y\in A}\sli_{i=1}^{N-1}\frac{1}{|y-a^*_i|^s}.
$$
In view of \eqref{nezadavali}, we get
$$
\inf_{y\in A}\sli_{i=1}^{N-1}\frac{1}{|y-a^*_i|^s} \geqslant c_2 N^{s/d}.
$$
Applying this estimate for $N$ instead of $N-1$, inequality \eqref{pspsps222} follows with $p_s=c_2$. \hfill \qed

\section{Some facts from the theory of pseudo-differential operators}\label{factspseudo}
In order to prove Theorem \ref{thequil} we need some facts from the theory of pseudo-differential operators that we will need. We give a brief introduction to the results we need in this section.

Let $\mathscr{S}(\R^d)$ be the class of Schwartz functions on $\R^d$ and $\mathscr{S}'(\R^d)$ be the set of tempered distributions. For an open set $\Omega$, we denote by $\mathscr{E}'(\Omega)$ the class of tempered distributions with compact support in $\Omega$. The Fourier transform is denoted by $\mathscr{F}$ and defined on $\mathscr{S}(\R^d)$ by the formula
$$
\mathscr{F}(f)(\xi):=\ili_{\R^d} f(x) e^{-2\pi i x \xi}dx, \;\; f\in \mathscr{S}(\R^d).
$$
We next introduce a class of functions (or {\it symbols}) that define standard pseudo-differential operators.
\begin{defin}
For a number $m\in \R$, we say that a function $p(x, \xi)\colon \Omega \times \R^d \to \R$ belongs to the class $S^m(\Omega)$ if $p\in C^\infty(\Omega \times \R^d)$ and for every compact set $K\subset \Omega$ and multi-indices $\alpha, \beta$ there exists a constant $C(K, \alpha, \beta)$ such that
\begin{equation}
|D^\alpha_\xi D^\beta_x p(x, \xi)| \leqslant C(K, \alpha, \beta) |\xi|^{m-|\alpha|}, \; \; x\in \Omega, \;\; |\xi|>1,
\end{equation}
where we use the notation
$$
D^\alpha_\xi p(x,\xi) := \frac{\partial^{|\alpha|}}{\partial \xi^\alpha} p(x, \xi), \;\; D^\beta_x p(x,\xi) := \frac{\partial^{|\beta|}}{\partial x^\beta} p(x, \xi).
$$
\end{defin}
The Paley--Schwartz--Wiener theorem implies that if $f\in \mathscr{E}'(\R^d)$, then its Fourier transform $\mathscr{F}(f)$ is a function with 
$$
|\mathscr{F}(f)(\xi)|\leqslant C(1  + |\xi|)^N, \; \; \xi \in \R^d
$$
for some positive constants $C$ and $N$. If $p$ belongs to $S^m(\Omega)$ and $f\in \mathscr{E}'(\Omega)$, then, for a fixed $x$, we can view $p(x, \xi)\mathscr{F}(f)(\xi)$ as a tempered distribution. We define an operator $P$ on $\mathscr{E}'(\Omega)$ by
\begin{equation}\label{defofpseudodiff}
P(f)(x):=\mathscr{F}^{-1}(p(x, \cdot)\mathscr{F}(f)(\cdot))(x), \; \; x\in \Omega.
\end{equation}
We further set
$$
\Psi^m(\Omega):=\{P\colon p\in S^m(\Omega)\}, \;\;\; \Psi^{-\infty}(\Omega):=\bigcap_{m\in \R}\Psi^m(\Omega).
$$

We continue with the definition of Sobolev spaces. For every $s\in \R$ and $p\in (1,\infty)$ set
$$
W^{s,p}_0(\Omega):=\{f\in \mathscr{E}'(\Omega)\colon \mathscr{F}^{-1}\left[(1+|\xi|^2)^{s/2}\cdot \mathscr{F}(f)(\xi)\right]\in L^p(\R^d)\}
$$
and
$$
W^{s, p}_{loc}=\{f\in \mathscr{S}'(\R^d)\colon \varphi f\in W^{s, p}_0(\R^d) \; \; \text{for any}\;\; \varphi\in C_0^\infty(\R^d)\}.
$$
As with the usual Sobolev spaces (i.e., with integer $s$), the following embedding property holds (see, e.g., \cite{Demengel} or \cite{DiNiezza}).
\begin{theorem}\label{sobolevembeddind}
Assume $\Omega$ is an open set in $\R^d$ with smooth boundary. If $sp>d$ and $f\in W^{s,p}_0(\Omega)$, then $f\in L^\infty(\Omega)$. 
\end{theorem}

The following theorem about the action of pseudo-differential operators on Sobolev spaces can be found in \cite[Theorem 2.1]{Treves} or \cite[Theorem 2.1D]{Taylor1991}.

\begin{theorem}\label{Treves21} 
If $P\in \Psi^m(\Omega)$ and $f\in W^{s,p}_0(\Omega)$, then $P(f)\in W^{s-m, p}_{loc}(\Omega)$. Moreover, if $P\in \Psi^{-\infty}(\Omega)$ and $f\in \mathscr{E}'(\Omega)$, then $P(f)\in C^\infty(\Omega)$. 
\end{theorem}

We further discuss regularity properties of solutions of the equation $Pu=f$. We say that the function $p\colon \Omega\times \R^d\to \R$ is {\it elliptic of order $m$} if $p\in S^m(\Omega)$ and for every $x\in \Omega$ there are two positive constants $c(x)$ and $r(x)$, such that
$$
|p(x, \xi)|\geqslant c(x)|\xi|^m, \; \; \mbox{for every $\xi$ with $|\xi|>r(x)$}.
$$

The following theorem can be found in \cite[Corollary 4.3]{Treves}.
\begin{theorem}\label{thparametrix}
Let $p$ be an elliptic function of order $m$ and $P\in \Psi^m(\Omega)$ be the corresponding operator defined as in \eqref{defofpseudodiff}. Then there exist $Q\in \Psi^{-m}(\Omega)$ and $R\in \Psi^{-\infty}(\Omega)$ such that 
$$
QP = I+R,
$$
where $I$ is the identity operator.
\end{theorem}

\section{Proof of Theorem \ref{thequil}}\label{sectionproofs2}
The case $s=d-1$ is done in \cite{sjogren1972}, thus we focus on the case $s<d-1$. Since $A\subset \R^{d+1}$ is $d$-regular at $x_0$ and $s\in (d-1, d)$, we obtain from Theorem \ref{coralwaysequal} that $U_s^\mu(x) = W_s(A)$ for any $x\in A\cap B(x_0, r_1)$ for some $r_1>0$. Since $A$ is $C^\infty$-smooth at $x_0$, there exists a $C^\infty$-smooth map $\psi\colon \tilde{B}\to B(x_0, r_0)$ such that $\tilde{B}\subset \R^d$ is open. Without loss of generality, we assume $r_0<r_1/2$. Set
\begin{equation}\label{temp321}
\textup{d}\mu^1:=\mathbbm{1}_{B(x_0, r_0)} \textup{d}\mu_s, \;\;\;\;\; \mu^2:=\mu_s-\mu^1,
\end{equation}
and 
$$
\nu:=\mu^1 \circ \psi.
$$
We notice that for $\tilde{x}\in \psi^{-1}(B(x_0, r_0/2))$ we have
$$
U^{\mu^1}_s(\psi(\t x)) = W_s(A)-U_s^{\mu^2}(\psi(\t x))
$$
and the right-hand side is a smooth function. Therefore, $U^{\mu^1}(\psi(\t x))\in C^\infty(\psi^{-1}(B(x_0, r_0/2)))$. 

We further write
\begin{equation}\label{glupiykorol}
U_s^{\mu^1}(\psi(\t x)) = \ili_{B(x_0, r_0)} \frac{\textup{d}\mu^1(y)}{|y-\psi(\t x)|^s} = \ili_{\t B} \frac{\textup{d}\nu(\t y)}{|\psi(\t y)-\psi(\t x)|^s}.
\end{equation}
Our next goal is to write the Taylor formula for $|\psi(\t y)-\psi(\t x)|^{-s}$ when $\t y$ is in the neighborhood of $\t x$. Since $\psi\in C^\infty$, there exists a $C^\infty$ matrix $a(\t x)$ and a $C^\infty$ vector-valued function $w_1(\t x, \t y)$ such that
$$
\psi(\t y)-\psi(\t x)=a(\t x)\cdot (\t y-\t x)+w_1(\t x, \t y)
$$
and for some constant $C$ and any component 
$$
|w_1(\t x, \t y)|\leqslant C|\t x - \t y|^2, \;\; \|\nabla_{\t x}\; w_1(\t x, \t y)\|_\infty \leqslant C|\t x - \t y|,
$$
where $\nabla_{\t x} w_1(\t x, \t y)$ is the matrix of gradients of $w_1$ in the first variable, and $\|\cdot \|_{\infty}$ is the $\ell^\infty$ matrix norm.
Therefore,
$$
|\psi(\t y)-\psi(\t x)|^2 = |a(\t x)\cdot (\t y - \t x)|^2 + w_2(\t x, \t y),
$$
where $w_2$ is a real-valued $C^\infty$ function with
$$
|w_2(\t x, \t y)|\leqslant C|\t x - \t y|^3, \; \; |\nabla_{\t x}\;w_2(\t x, \t y)|\leqslant C|\t x - \t y|^2.
$$ 
If $r_0$ is small enough and $\t y, \t x\in B(x_0, r_0/2)$, then
$$
 \left| \frac{w_2(\t x, \t y)}{|a(\t x)\cdot (\t y - \t x)|^2}\right| \leqslant 1/2.
$$ Consequently,
\begin{equation}\label{trubadur}
|\psi(\t y)-\psi(\t x)|^{-s}=|a(\t x)\cdot (\t y - \t x)|^{-s}\cdot \left(1+ \frac{w_2(\t x, \t y)}{|a(\t x)\cdot (\t y - \t x)|^2} \right)^{-s/2}.
\end{equation}
We notice that
$$
w_3(\t x, \t y):=\frac{w_2(\t x, \t y)}{|a(\t x)\cdot (\t y - \t x)|^2} \in C^1
$$
with $|\nabla_{\t x}\; w_3(\t x, \t y)|$ bounded. Therefore, \eqref{trubadur} implies
$$
|\psi(\t y)-\psi(\t x)|^{-s} = |a(\t x)\cdot (\t y - \t x)|^{-s} + w_4(\t x, \t y),
$$
where 
$$
|w_4(\t x, \t y)|\leqslant C_1 |\t y - \t x|^{-s+1}, \; \; \; \; |\nabla_{\t x}\; w_4(\t x, \t y)|\leqslant C_2|\t y - \t x|^{-s}.
$$ 
We plug this into \eqref{glupiykorol} to get
$$
U_s^{\mu_1}(\psi(\t x)) = \ili_{\t B}\frac{\textup{d}\nu(\t y)}{|a(\t x)\cdot (\t y - \t x)|^{s}} + \ili_{\t B}w_4(\t x, \t y) \textup{d}\nu(\t y).
$$
Since 
$$
\ili_{\t B}|\nabla_{\t x} w_4(\t x, \t y)|\textup{d}\nu(\t y) \leqslant C_2 \ili_{\t B} \frac{\textup{d}\nu(\t y)}{|\t y - \t x|^s} \leqslant C_3 \ili_{B(x_0, r_0)}\frac{\textup{d}\mu(y)}{|y - x|^s} \leqslant C_3 W_s(A),
$$
we see that the function $\t x \mapsto \ili_{\t B}w_4(\t x, \t y) \textup{d}\nu(\t y)$ belongs to $W^{1, \infty}(\psi^{-1}(B(x_0, r_0/4)))$.
Let $u$ be a Schwartz function equal to $1$ in $\psi^{-1}(B(x_0, r_0/4))$ and to $0$ outside of $\psi^{-1}(B(x_0, r_0/2))$. Then
\begin{equation}\label{buratino}
u(\t x) \ili_{\t B}\frac{\textup{d}\nu(\t y)}{|a(\t x)\cdot (\t y - \t x)|^{s}} = u(\t x)U_s^{\mu_1}(\psi(\t x))-u(\t x) \ili_{\t B}w_4(\t x, \t y) \textup{d}\nu(\t y) =: w(\t x) \in W_0^{1, \infty}(\R^d).
\end{equation}
We next show that the operator
\begin{equation}\label{karabas}
P\colon \nu \mapsto u(\t x) \ili_{\t B}\frac{\textup{d}\nu(\t y)}{|a(\t x)\cdot (\t y - \t x)|^{s}}
\end{equation}
is pseudo-differential. Namely, we use the Plancherel identity to obtain
\begin{equation}\label{duremar}
\ili_{\t B}\frac{\textup{d}\nu(\t y)}{|a(\t x)\cdot (\t y - \t x)|^{s}} = \ili_{\R^d} \mathscr{F}(\nu)(\xi) \overline{\mathscr{F}_{\t y}(|a(\t x)\cdot (\t y - \t x)|^{-s})(\xi)} \textup{d}\xi.
\end{equation}
By definition of the Fourier Transform, we have
$$
\mathscr{F}_{\t y}(|a(\t x)\cdot (\t y - \t x)|^{-s})(\xi) = \ili_{\R^d} |a(\t x)\cdot (\t y - \t x)|^{-s} e^{-2\pi i \t y \xi}\textup{d}\t y.
$$
Since the matrix $a(\t x)$ is a $d\times (d+1)$ matrix of rank $d$, we observe that the set $\{a(\t x)\cdot \t y\colon \t y \in \R^d\}$ is a $d$-dimensional linear subspace of $\R^{d+1}$. Take a rotation $R$ that maps this set to $\{y=(y(1), \ldots, y(d+1))\in \R^{d+1}\colon y(d+1)=0\}$ and an operator $T$ that maps the latter space to $\R^d$ by erasing the $(d+1)$'st coordinate. We make a change of variables 
$$
\t z=T\cdot R \cdot a(\t x)\cdot (\t y - \t x).
$$
By definition of $T$ and $R$, we have
$$
|\t z| = |T\cdot R \cdot a(\t x)\cdot (\t y - \t x)| = |a(\t x)(\t y - \t x)|,
$$
and therefore, setting $b(\t x):=(T\cdot R\cdot a(\t x))^{-1}$, we get
\begin{multline}\label{chichi}
\mathscr{F}_{\t y}(|a(\t x)\cdot (\t y - \t x)|^{-s})(\xi) = \ili_{\R^d} |a(\t x)\cdot (\t y - \t x)|^{-s} e^{-2\pi i \t y \xi}\textup{d}\t y \\
= 
e^{-2\pi i \t x \xi}\ili_{\R^d} |\t z|^{-s} e^{-2 \pi i (b(\t x)\t z) \xi}|\det(b(\t x))|\textup{d}\t z = |\det(b(\t x))|e^{-2\pi i \t x \xi} \mathscr{F}(|\t z|^{-s})((b^t(\t x)) \xi)\\
=|\det(b(\t x))|e^{-2\pi i \t x \xi}  |b^t(\t x) \xi|^{s-d}.
\end{multline}
We plug \eqref{chichi} into \eqref{duremar}:
\begin{multline}
\ili_{\t B}\frac{\textup{d}\nu(\t y)}{|a(\t x)\cdot (\t y - \t x)|^{s}} = \ili_{\R^d} \mathscr{F}(\nu)(\xi)|\det(b(\t x))|\cdot  |b^t(\t x) \xi|^{s-d} e^{2\pi i \t x \xi} \textup{d}\xi\\
 = \mathscr{F}^{-1}\Big(\mathscr{F}(\nu)(\xi)|\det(b(\t x))|\cdot  |b^t(\t x) \xi|^{s-d}\Big)(\t x).
\end{multline}
Setting 
$$
p(\t x, \xi):=u(\t x)|\det(b(\t x))|\cdot  |b^t(\t x) \xi|^{s-d},
$$
we obtain that the operator $P$ defined in \eqref{karabas} is an elliptic pseudo-differential with symbol $p\in S^{s-d}(\tilde{B})$. We apply Theorem \ref{thparametrix} to  equation \eqref{buratino}. Since $P\nu = w$, we get
\begin{equation}\label{temptemp}
\nu + R\nu = Qw, \; \; \; \; R\nu \in C^\infty(\tilde{B}).
\end{equation}
Further, since $w\in W^{1, \infty}_0(\tilde B)$, we get from Theorem \ref{Treves21} that $Qw\in W^{1+s-d, p}_{loc}(\tilde {B})$ for any $p>1$. By Theorem \ref{sobolevembeddind}, we obtain that $Qw \in L^\infty\Big(\psi^{-1}(B(x_0, r_0/4))\Big)$, and from \eqref{temptemp} we get $\nu\in L^\infty\Big(\psi^{-1}(B(x_0, r_0/4))\Big)$. Since the measure $\mu_1$ defined in \eqref{temp321} is an image of $\nu$ under a smooth map $\psi^{-1}$, we deduce that for $r<r_0/4$
$$
\mu(B(x_0, r))= \nu(\psi^{-1}(B(x_0, r))) \leqslant C_1 \mathcal{H}_d(\psi^{-1}(B(x_0, r))) \leqslant C_2 r^d.
$$
\hfill \qed

\bibliography{ref}
\bibliographystyle{plain}

\end{document}